\documentclass[12pt,leqno]{article}
\usepackage{amsfonts,amsthm,amsmath}
\newtheorem{thm}{Theorem}
\newtheorem{prop}[thm]{Proposition}
\newtheorem{lem}[thm]{Lemma}

\theoremstyle{remark}
\newtheorem{rem}[thm]{Remark}
\newtheorem{Q}[thm]{Question}

\newcommand{\1}{\mathbf{1}}

\newcommand{\cD}{\mathcal{D}}

\DeclareMathOperator{\supp}{supp}

\DeclareMathOperator{\wt}{wt}

\begin{document}

\title{
On a 5-design related to a putative extremal doubly even 
self-dual code of length a multiple of 24}

\author{
Masaaki Harada\thanks{
Research Center for Pure and Applied Mathematics,
Graduate School of Information Sciences,
Tohoku University, Sendai 980--8579, Japan.
email: mharada@m.tohoku.ac.jp.
This work 
was partially carried out
at Yamagata University.}
}

\maketitle

\begin{abstract}
By the Assmus and Mattson theorem,
the codewords of each nontrivial weight
in an extremal doubly even self-dual code of length $24m$
form a self-orthogonal $5$-design.  
In this paper, 
we study the codes constructed from self-orthogonal $5$-designs 
with the same parameters as the above $5$-designs.
We give some parameters of a self-orthogonal $5$-design
whose existence is equivalent to that of 
an extremal doubly even self-dual code of length $24m$
for $m=3,4,5,6$.
If $m \in \{1,\ldots,6\}$, $k \in \{m+1,\ldots,5m-1\}$
and $(m,k) \ne (6,18)$, then it is shown that
an extremal doubly even self-dual code of length $24m$ 
is generated by codewords of weight $4k$.
\end{abstract}



\section{Introduction}\label{sec:1}

A doubly even self-dual code of length $n$
exists if and only if $n$ is divisible by $8$.
The minimum weight $d(C)$ of a doubly even self-dual code $C$
of length $n$
is bounded above by $d(C) \le 4 \lfloor n/24\rfloor+4$~\cite{MS73}.
A doubly even self-dual code meeting the bound is called
{\em extremal}.
In case that $n \equiv 0 \pmod{24}$, the only known extremal doubly even
self-dual codes are the extended Golay code and the
extended quadratic residue code of length $48$.
The existence of an extremal doubly even self-dual
code of length $72$ is a long-standing open question~\cite{Sloane72}.

A $t$-$(v,k,\lambda)$ design is called {\em self-orthogonal} if
the block intersection numbers have the same parity
as the block size $k$ (see~\cite{T86}).
If $\cD$ is a self-orthogonal $t$-$(v,k,\lambda)$ design 
with $k$ even, then the code $C(\cD)$, which is 
generated by the rows of an incidence matrix of $\cD$,
is a self-orthogonal code.
By the Assmus and Mattson theorem~\cite{AM},
the supports of the codewords of weight $4k \ (\ne 0,24m)$ in
an extremal doubly even self-dual code of length $24m$
form a self-orthogonal $5$-design.
We denote the parameters of the design by
$5$-$(24m,4k,\lambda_{24m,4k})$.
Then, throughout this paper, we denote any self-orthogonal
$5$-$(24m,4k,\lambda_{24m,4k})$ design by $\cD_{24m,4k}$.
That is, $\cD_{24m,4k}$ is 
a self-orthogonal $5$-design 
with the same parameters as the self-orthogonal
$5$-design formed from the supports of the codewords of 
weight $4k$ in an extremal doubly even self-dual code
of length $24m$.
This gives rise to a natural question, namely,
is the code $C(\cD_{24m,4k})$
always an extremal doubly even self-dual code?

It is well known that $C(\cD_{24,8})$ is the extended Golay code
(see~\cite[Theorem 8.6.2]{AK-book}).
It was shown that $C(\cD_{24m,4m+4})$ ($m =2,3,4$) is an extremal doubly even 
self-dual code~\cite{HMT05,HKM04,H05}, respectively.
This means that the existence of
an extremal doubly even self-dual code
of length $24m$ ($m =1,2,3,4$) is equivalent to that of
a self-orthogonal
$5$-$(24m,4k,\lambda_{24m,4k})$ design, where
$(4k,\lambda_{24m,4k})=(8,  1)$,
$(12, 8 )$,
$(16, 78 )$ and
$(20, 816 )$, respectively. 
The powerful tool which is used in~\cite{HKM04,HMT05} 
is the use of fundamental equations, sometimes called the 
Mendelsohn equations~\cite{M71} (see also e.g., \cite{T86}),
obtained by counting the number of blocks  
that meet $S$ in $i$ points for some subset $S$ of the point set.
The approach in~\cite{H05} is also similar to that 
in~\cite{HKM04,HMT05} except that 
Gleason's theorem (see~\cite{MS73}) is employed 
to obtain stronger consequences.

In this paper, we study self-orthogonal $5$-designs $C(\cD_{24m,4k})$
for $k \in \{m+2,\ldots,5m-1\}$, which are
related to codewords of weight other than the minimum weight.
More precisely,  we consider a problem whether
$C(\cD_{24m,4k})$ is an extremal doubly even self-dual code or not
for $m \in \{1,\ldots,6\}$ and $k \in \{m+2,\ldots,5m-1\}$.
In addition to the above approach done in~\cite{H05,HKM04,HMT05},
it is useful in this paper
to observe weight enumerators of $C(\cD_{24m,4k})$ and its dual codes,
and singly even self-dual codes containing 
$C(\cD_{24m,4k})$ and their shadows.
As a summary, in Table~\ref{Tab:D}\footnote{See Sections~\ref{sec:SD}
and~\ref{sec:min} for the marks $*$ in Table~\ref{Tab:D}.},
we list some partial answers to the above problem
for $m \in \{1,\ldots,6\}$ and $k \in \{m+1,\ldots,3m\}$.
For the cases $(24m,4k)$ that $C(\cD_{24m,4k})$ is self-dual,
we list ``Yes'' in the second column of Table~\ref{Tab:D}.
When $C(\cD_{24m,4k})$ is self-dual,
we list ``Yes''  in the third column
in case that $C(\cD_{24m,4k})$ is extremal.
We also list the possible 
minimum weights, when $C(\cD_{24m,4k})$ is self-dual but
not extremal.
It is shown that
$C({\cD}_{24m,4k}) = C({\cD}_{24m,24m-4k})$
for $m \in \{1,\ldots,6\}$ and $k \in \{m+1,\ldots,3m-1\}$
(Proposition~\ref{prop:CD}).

The main results of this paper are the following theorems.

\begin{thm}\label{thm:main1}
Suppose that $(24m,k,\lambda)$ is each of the following:
\begin{align*}
&
(72, 24, 1406405 ), (72, 32, 238957796 ),
\\ &
(96,36, 28080500448 ), 
(96,44, 1167789832440 ),
\\ &
(120, 56, 5156299310025435),
(144, 68, 21788133027489299328).
\end{align*}
Then the existence of a self-orthogonal
$5$-$(24m,k,\lambda)$ design is equivalent to that of
an extremal doubly even self-dual code
of length $24m$.
\end{thm}

\begin{thm}\label{thm:main2}
Suppose that $m \in \{1,\ldots,6\}$ and
$k \in \{m+1,\ldots,5m-1\}$.
If $(m,k) \ne (6,18)$, then
an extremal doubly even self-dual code of length $24m$ 
is generated by codewords of weight $4k$.
\end{thm}
\begin{rem}
For some cases $(m,k)$, the above theorem is already 
known (see Table~\ref{Tab:D}).
It is still unknown whether $C(\cD_{144, 72})$ is self-dual
or not (see Remark~\ref{rem}).
\end{rem}

\begin{table}[p]
\caption{Codes $C(\cD_{24m,4k})$ ($m=1,\ldots,6$, $k=m+1,\ldots,3m$)}
\label{Tab:D}
\begin{center}
{\footnotesize
\begin{tabular}{l|c|c|c} 
\noalign{\hrule height0.8pt}
\multicolumn{1}{c|}{Parameters of $\cD_{24m,4k}$}  &
Self-dual  & Extremal & Ref.\\
\hline
$( 24,  8,  1)$  &Yes & Yes & (see~\cite{AK-book}) \\
$( 24, 12, 48)$  &Yes & Yes &\cite{T86} \\
\hline
$( 48, 12, 8 )$     &Yes & Yes &\cite{HMT05}\\
$( 48, 16, 1365 )$  &Yes & Yes &\cite{D}\\
$( 48, 20, 36176 )$ &Yes & Yes &\cite{D}\\
$( 48, 24, 190680 )$&Yes &  8, 12  \\
\hline
$( 72, 16, 78 )$       &Yes & Yes &\cite{HKM04}\\
$( 72, 20, 20064 )$    &Yes & 12, 16 &\cite{D}\\
$( 72, 24, 1406405 )$  &Yes & Yes$^*$\\
$( 72, 28, 30888000 )$ &Yes$^*$ & 12, 16          & \\
$( 72, 32, 238957796 )$&Yes & Yes$^*$\\
$( 72, 36, 693996160 )$&Yes & 12, 16 &\cite{D}\\
\hline
$( 96, 20, 816 )$          & Yes & Yes &\cite{H05}\\
$( 96, 24, 257180 )$       &Yes & 16, 20&\cite{D}\\
$( 96, 28, 29975400 )$     &Yes & 12, 20$^*$\\
$( 96, 32, 1390528685 )$   &Yes & 12, 16, 20 &\cite{D}\\
$( 96, 36, 28080500448 )$  &Yes & Yes$^*$\\
$( 96, 40, 261513764460 )$ &Yes & 12, 16, 20 &\cite{D}\\
$( 96, 44, 1167789832440 )$&Yes & Yes$^*$\\
$( 96, 48, 2561776811880 )$&Yes$^*$ & 12, 16, 20 & \\
\hline
$( 120, 24, 8855 )$  &         Yes &  16, 24 &\cite{CW12}\\
$( 120, 28, 3146400 )$         &Yes & 16, 20, 24 &\\
$( 120, 32, 502593700 )$       &Yes & 12, 16, 24$^*$&\\
$( 120, 36, 37237713920 )$     &Yes$^*$ & 12--24&\\
$( 120, 40, 1372275835848 )$   &Yes$^*$ & 12, 24$^*$&\\
$( 120, 44, 26386953577600 )$  &Yes$^*$ & 12--24&\\
$( 120, 48, 274320081834480 )$ &Yes$^*$ & 12, 24$^*$&\\
$( 120, 52, 1582247888524800 )$&Yes$^*$ & 12--24 &\\
$( 120, 56, 5156299310025435 )$&Yes & Yes$^*$&\\
$( 120, 60, 9606041207517888 )$&Yes$^*$ & 12--24 &\\
\hline
$( 144, 28, 98280 )$               &Yes & 16, 20, 28 &\cite{HMM}\\
$( 144, 32, 37756202 )$            &Yes & 16--28&\\
$( 144, 36, 7479335776 )$          &Yes & 16, 20, 28$^*$&\\
$( 144, 40, 765322879032 )$        &Yes & 12--28&\\
$( 144, 44, 42785304274536 )$      &Yes & 12, 16, 20, 28$^*$&\\
$( 144, 48, 1359454757387265 )$    &Yes & 12--28&\\
$( 144, 52, 25319185698144240 )$   &Yes & 12, 16, 28$^*$&\\
$( 144, 56, 283096123959568608 )$  &Yes$^*$ & 12--28&\\ 
$( 144, 60, 1935608752827917264 )$ &Yes & 12, 28$^*$&\\
$( 144, 64, 8205989047403924124 )$ &Yes & 12--28&\\
$( 144, 68, 21788133027489299328 )$&Yes & Yes$^*$&\\
$( 144, 72, 36470135955078919440 )$& ?  & --\\
\noalign{\hrule height0.8pt}
\end{tabular}
}
\end{center}
\end{table}

\section{Preliminaries}\label{sec:Pre}

\subsection{Self-dual codes and shadows}

In this paper, codes mean binary codes.
A code is called {\em doubly even} if 
every codeword has weight a multiple of $4$.
A code $C$ is called
\textit{self-orthogonal} if $C \subset C^{\perp}$,
and $C$ is  called
\textit{self-dual} if $C = C^{\perp}$,
where
$C^{\perp}$ is the dual code of $C$ under the
standard inner product.
A self-dual code which is not doubly even
is called
{\em singly even}, namely, 
a  singly even self-dual code contains a
codeword of weight $\equiv 2
\pmod 4$.
It is  known that a self-dual code of length $n$ exists
if and only if  $n$ is even, and
a doubly even self-dual code of length $n$
exists if and only if $n$ is divisible by eight.
The minimum weight $d(C)$ of a doubly even
self-dual code $C$ of length $n$
is bounded by
$d(C)  \le 4  \lfloor n/24 \rfloor + 4$~\cite{MS73}.
A doubly even self-dual code meeting the bound is called  {\em extremal}.
In case that $n \equiv 0 \pmod{24}$, the only known extremal doubly even
self-dual codes are the extended Golay code and the
extended quadratic residue code of length $48$.
The existence of an extremal doubly even self-dual
code of length $72$ is a long-standing open 
question~\cite{Sloane72}. 

Let $C$ be a singly even self-dual code and let $C_0$ denote the
subcode of codewords having weight $\equiv0\pmod4$. Then $C_0$ is
a subcode of codimension $1$. The {\em shadow} $S$ of $C$ is
defined to be $C_0^\perp \setminus C$. 
Shadows were introduced by Conway and Sloane~\cite{C-S}, in order
to
provide restrictions on the weight enumerators of singly even
self-dual codes
(see~\cite{C-S} for fundamental results on shadows).
Let $D$ be a doubly even code of length $n \equiv 0 \pmod 8$.
Suppose that $D$ has dimension $n/2-1$ and
$D$ contains the all-one vector $\1$.
Then there are three self-dual codes
lying between $D^\perp$ and $D$, 
one of which is singly even and the others are doubly even
(see~\cite{MST}).

\subsection{Self-orthogonal designs and Mendelsohn equations} 

A $t$-$(v,k,\lambda)$ design $\cD$ is a set $X$ of $v$ points
together with a collection of $k$-subsets of $X$ (called blocks)
such that every $t$-subset of $X$ is contained in exactly
$\lambda$ blocks.
A $t$-design with no repeated block is called {\em simple}.
In this paper, designs mean simple designs.
It follows that
every $i$-subset of points ($i\le t$) is contained in exactly
$\lambda_i =\lambda {\binom{v-i}{t-i}}/{\binom{k-i}{t-i}}$ blocks.
The number $\lambda_1$ is traditionally denoted by $r$, and the total number of
blocks is $b=\lambda_0$.
A $t$-design can be represented by its (block-point)
incidence matrix $A=(a_{ij})$, where $a_{ij}=1$ if the $j$th
point is contained in the $i$th block
and $a_{ij}=0$ otherwise.

The {\em block intersection numbers} of a $t$-$(v,k,\lambda)$ design
are the cardinalities of
the intersections of any two distinct blocks.
A $t$-$(v,k,\lambda)$ design 
is called {\em self-orthogonal} if
the block intersection numbers have the same parity
as the block size $k$ (see~\cite{T86}).
The term self-orthogonal is due to a natural connection between 
such designs and self-orthogonal codes.
Throughout this paper, we denote the code generated by the rows of an 
incidence matrix of $\cD$ by $C(\cD)$.
If $\cD$ is a self-orthogonal $t$-$(v,k,\lambda)$ design
with $k$ even, then $C(\cD)$
is a self-orthogonal code.

Let $\cD$ be a $t$-$(v,k,\lambda)$ design.
Let $v\in C(\cD)^{\perp}$ be a vector of weight $w>0$.
Denote by $n_i$ the number of rows of an incidence matrix of
$\cD$ intersecting exactly 
$i$ positions of the support of $v$ in ones. Then 
we have the system of equations:
\begin{equation}
\label{eq:M}
\sum_{i=0}^{\min\{k,w\}}\binom{i}{j}n_{i}
=\lambda_{j}\binom{w}{j}
\quad(j=0,1,\ldots,t).
\end{equation}
These fundamental equations, which 
are sometimes called Mendelsohn equations~\cite{M71}
(see also~\cite{T86}),
are the powerful tool in the study of this paper.
We note that $n_i=0$ if $i$ is odd, $i > k$ or $i > w$. 

The following lemma follows immediately.

\begin{lem}\label{lem:basic}
Let $\cD$ be a self-orthogonal $t$-$(v,k,\lambda)$ design
with $k \equiv 0 \pmod 4$.
\begin{itemize}
\item[\rm (i)]
If the system of equations {\rm (\ref{eq:M})} has
no solution $(n_0,n_2,\ldots)$ consisting of nonnegative
integers for some $w$, then
$C(\cD)^\perp$ contains no vector of weight $w$.
\item[\rm (ii)]
If the system of equations {\rm (\ref{eq:M})} has
no solution $(n_0,n_2,\ldots)$ consisting of nonnegative
integers for each $w$ with $0 < w < v$, 
$w \not \equiv 0 \pmod 4$, then
$C(\cD)$ is doubly even self-dual.
\end{itemize}
\end{lem}

The complementary design $\overline{\cD}$ of a design $\cD$ is
obtained by replacing each block of $\cD$ by its complement.
The following lemma is used in Section~\ref{sec:min}
to show that
$C({\cD}_{24m,4k}) = C({\cD}_{24m,24m-4k})$
for $m \in \{1,\ldots,6\}$ and $k \in \{m+1,\ldots,3m-1\}$.

\begin{lem}\label{lem:CD}
Let $\cD$ be a self-orthogonal $t$-$(v,k,\lambda)$
design with $k$ even.
Suppose that $C(\cD)$ is self-dual.
Then $C(\cD)=C(\overline{\cD})$ if $\1\in C(\overline{\cD})$,
and 
$C(\overline{\cD}) \subset C(\cD)$ with
$|C(\cD):C(\overline{\cD})|=2$ otherwise.
\end{lem}
\begin{proof}
Since $C(\cD)$ is self-dual, $\1 \in C(\cD)$.
It turns out that $C(\overline{\cD}) \subseteq C(\cD)$
and $\langle C(\overline{\cD}), \1 \rangle = C(\cD)$.
The result follows.
\end{proof}


\section{On the self-duality}\label{sec:SD}

In this section, 
we describe how to determine the self-duality
given in the second column of Table~\ref{Tab:D}
for the cases denoted by $*$ in Table~\ref{Tab:D}.
For the other cases, 
the self-duality is determined by 
Lemma~\ref{lem:basic} (ii) only.

\begin{prop}\label{prop:SD-P1}
The codes $C(\cD_{72,28})$,
$C(\cD_{96,48})$, $C(\cD_{120,60})$ and $C(\cD_{120,52})$
are self-dual.
\end{prop}
\begin{proof}
All cases are similar, and we only give the details
for $C(\cD_{72,28})$.

Note that $\cD_{72,28}$ has  the following parameters:
\begin{multline*}
\lambda_{0} = 4397342400, 
\lambda_{1} = 1710077600, 
\lambda_{2} = 650311200, 
\\
\lambda_{3} = 241544160, 
\lambda_{4} = 87516000, 
\lambda_{5} = 30888000.
\end{multline*}
Let ${v}\in C(\cD_{72,28})^{\perp}$ be a vector of weight $w >0$.
For each $w$ of the cases
with $w \equiv 1 \pmod 2$ and $w \le 8$, the system of 
equations (\ref{eq:M}) has no solution.
In addition, for $w=10$, (\ref{eq:M}) has the following unique solution:
\begin{multline*}
n_0 = 41076475, 
n_2 = 1096595775, 
n_4 = 2375199750, 
\\
n_6 = 834337350,
n_8 = 50284575, 
n_{10}= -151525.
\end{multline*}
Hence, 
there is no vector of weights $2,4,6,8,10$ in $C(\cD_{72,28})^{\perp}$.
The number $\lambda_0$ of blocks satisfies
that $2^{32} < \lambda_0 < 2^{33}$.
Therefore, $C(\cD_{72,28})^{\perp}$ is an even code such that
the minimum weight is at least $12$ and
the dimension is at most $39$.

Let $D_{72}$ be a doubly even code of length $72$ satisfying the conditions
that $D_{72}$ has dimension $\ell \in \{33,34,35,36\}$,
both $D_{72}$ and $D_{72}^\perp$ have minimum weights 
at least $12$ and $\1 \in D_{72}$.
We denote the weight enumerators
of $D_{72}$ and $D_{72}^\perp$ by $W_{D_{72}}$ and 
$W_{D_{72}^\perp}$, respectively.
In this case, 
$W_{D_{72}}$ can be written as:
\begin{align*}
&
x^{72} + a x^{60} y^{12} + b x^{56} y^{16} + c x^{52} y^{20} 
+ d x^{48} y^{24} + e x^{44} y^{28} 
+ f x^{40} y^{32} 
\\&
+ (2^\ell -2 -2 a - 2 b - 2 c - 2 d - 2 e - 2 f) x^{36} y^{36} 
+\cdots + y^{72},
\end{align*}
using nonnegative integers $a,b,c,d,e,f$.
Set $W_{D_{72}^\perp} = \sum_{i=0}^{72} B_i x^{72-i}y^i$.
By the MacWilliams identity, we have:
\begin{align*}
2^{\ell}B_2 =
& 2^6(\chi_{2,\ell} + 36 a + 25 b + 16 c + 9 d + 4 e + f),\\
2^{\ell}B_4 =&2^6(\chi_{4,\ell} + 5640 a + 2450 b + 800 c + 114 d 
  - 56 e - 30 f),\\
2^{\ell}B_6 =& 
2^6(\chi_{6,\ell} + 313060 a + 77385 b + 8976 c - 1223 d + 196 e + 433 f),\\
2^{\ell}B_8 =& 
2^6(\chi_{8,\ell} + 7582080 a + 811360 b - 43520 c - 5280 d  + 1408 e - 4000 f),\\
2^{\ell}B_{10} =& 
2^6(\chi_{10,\ell} + 86892960 a + 887656 b - 372096 c + 100584 d - 17248 e 
\\ &+ 26536 f),
\end{align*}
where $(\chi_{2i,33},\chi_{2i,34},\chi_{2i,35})$ are as follows:
\begin{align*}
&(-4831838127    , -9663676335     ,  -19327352751), \\
&(84557200770    , 169114369410    ,  338228706690), \\
&(-958309695231  , -1916624273151  ,  -3833253428991), \\
&(7906469297760  , 15812564565600  ,  31624755101280), \\
&(-50582253079512, -101181262793688,  -202379282222040),
\end{align*}
for $i=1,2,3,4,5$, respectively.

The assumptions $B_{2i}=0$ $(i=1,2,3,4,5)$ yield
the following:
\[
b = \alpha_\ell - 12 a,
c = \beta_\ell + 66 a, 
d = \gamma_\ell - 220 a,
e = \delta_\ell + 495 a,
f = \varepsilon_\ell - 792 a,
\]
where 
\begin{align*}
(\alpha_\ell,\beta_\ell,\gamma_\ell,\delta_\ell,\varepsilon_\ell)=&
(30105,2273040,57830955,549766080,2075173947), \\
&(61497,4534992,115706955,1099419840,4150537083),\\
&(124281,9058896,231458955,2198727360,8301263355),
\end{align*}
for $\ell=33,34,35$, respectively.
For $\ell=33,34,35$,
it follows from $b \ge 0$ that
\[
e =\delta_\ell + 495a 
\le \delta_\ell + \frac{165}{4} \alpha_\ell <  4397342400 =\lambda_0.
\]
Since 
$C(\cD_{72,28})$ contains at least $4397342400$ codewords of 
weight $28$,
we obtain a contradiction.
Therefore, $C(\cD_{72,28})$ must be self-dual.
\end{proof}


\begin{prop}\label{prop:SD-P2}
The codes $C(\cD_{120, 36})$, $C(\cD_{120, 40})$, $C(\cD_{120, 44})$, 
$C(\cD_{120, 48})$ and $C(\cD_{144, 56})$ are self-dual.
\end{prop}
\begin{proof}
All cases are similar, and we only give the details
for $C(\cD_{120, 40})$.

Note that $\cD_{120, 40}$ has  the following parameters:
\begin{multline*}
\lambda_{0} = 397450513031544, 
\lambda_{1} = 132483504343848, 
\lambda_{2} = 43418963608488, 
\\
\lambda_{3} = 13982378111208,
\lambda_{4} = 4421777693288, 
\lambda_{5} = 1372275835848.
\end{multline*}
Let ${v}\in C(\cD_{120, 40})^{\perp}$ be a vector of weight $w >0$.
For each $w$ of the cases with
$w \equiv 1 \pmod 2$ and $w \le 8$, 
the system of equations (\ref{eq:M}) has no solution.
The number $\lambda_0$ of blocks satisfies
that $2^{48} < \lambda_0 < 2^{49}$.
Hence, $C(\cD_{120, 40})^{\perp}$ is an even code such that
the minimum weight is at least $10$ and
the dimension is at most $71$.

Let $D_{120}$ be a doubly even code of length $120$ satisfying the conditions
that $D_{120}$ has dimension $\ell \in \{49,\ldots,60\}$, 
$D_{120}$ has minimum weight at least $12$,
$D_{120}^\perp$ has minimum weight at least $10$ 
and $\1 \in D_{120}$.
We show that $\ell \ne 49,50,\ldots,59$
in the following two steps.

The first step shows that $\ell \ne 49,\ldots,58$.
The approach is similar to that 
given in Proposition~\ref{prop:SD-P1}.
Suppose that $\ell \in \{49,\ldots,58\}$.
Then,
by considering the possible weight enumerators of $D_{120}$ 
and $D_{120}^\perp$,
one can obtain a contradiction for each $\ell$.
Since the situation is more complicated than that for
$C(\cD_{72,28})$ considered in Proposition~\ref{prop:SD-P1},
we omit the details to save space. 
We remark that this argument does not work to show
that $\ell \ne 59$.

The second step shows that $\ell \ne 59$.
The approach is to consider singly even
self-dual codes containing $D_{120}$.
Suppose that $\ell=59$.
Since $D_{120}$ contains $\1$,
there are three self-dual codes
lying between $D_{120}^\perp$ and $D_{120}$, 
one of which is singly even and the others are doubly even
(see~\cite{MST}).
We denote the singly even code by $C_{120}$, noting that
$D_{120}$ is the subcode $(C_{120})_0$ consisting of codewords
of weight $\equiv 0 \pmod 4$ of $C_{120}$.
Let $S_{120}$ be the shadow of $C_{120}$.
Since the weight of a vector in $S_{120}$ is divisible by four
\cite{C-S} and $D_{120}^{\perp}$ has minimum weight at least $10$,
$C_{120}$ and $S_{120}$ 
have minimum weights at least $10$ and $12$, respectively.
Using~\cite[(10) and (11)]{C-S}, from the condition on the 
minimum weights,
one can determine the possible weight enumerators 
$\sum_{i=0}^{120} A_i x^{120-i}y^i$
and 
$\sum_{i=0}^{120} B_i x^{120-i}y^i$
of $C_{120}$ and $S_{120}$, respectively.
In this case, the possible weight enumerators 
can be written using integers $a,b,c,d,e,f,g,h$.

We investigate the number of codewords of weight $40$.
In this case, we have that
\begin{multline*}
A_{40}=
198725556937080 + 32980992a - 28160b - 15504c \\ 
+ 4896d + 161525e - 599494f - 4385880g + 91345008h.
\end{multline*}
Using the mathematical software {\sc Mathematica},
we have verified that
$A_{2i} \ge 0$ $(i=5,\ldots,16)$ and
$B_{4i} \ge 0$ $(i=3,\ldots,9)$
yield
\[
A_{40} <397450513031544 = \lambda_0,
\]
where
$A_{2i}$ $(i=5,\ldots,16)$ and
$B_{4i}$ $(i=3,\ldots,9)$ are listed in 
Tables~\ref{Tab:120-A} and \ref{Tab:120-B}, respectively.
Since $C(\cD_{120,40})$ contains at least $397450513031544$ codewords of 
weight $40$, we obtain a contradiction.
Therefore, $C(\cD_{120,40})$ must be self-dual.
This completes the proof.
\end{proof}

\begin{table}[thb]
\caption{Weight enumerator of $C_{120}$}
\label{Tab:120-A}
\begin{center}
{\scriptsize
\begin{tabular}{c|l} 
\noalign{\hrule height0.8pt}
 $i$ & \multicolumn{1}{c}{$A_{i}$} \\
\hline
{10}&$h$\\
{12}&$g + 30h$\\
{14}&$f + 24g + 425h$\\
{16}&$e + 18f + 264g + 3760h$\\
{18}&$d + 12e + 139f + 1736g + 23100h$\\
{20}&$c + 6d + 50e + 564f + 7380g + 103256h$\\
{22}&$64b - 3d + 28e +  1009f + 19800g + 339180h + 26391755$\\
{24}&$4096a - 384b - 20c - 88d - 441e - 1218f + 25080g + 789840h$\\
{26}&$265912320 - 49152a - 64b - 102d - 1288e- 10717f - 35640g + 1096410h$\\
{28}&$2968094880 + 221184a + 4864b + 190c + 564d+ 364e  - 20424f
     - 238590g - 118980h$\\
{30}&$29559455744 - 311296a - 6720b + 1210d+ 7800e  + 7631f 
    - 473880g - 4961862h$\\
{32}&$238259763105 - 946176a - 25984b - 1140c - 1944d + 9971e + 103766f 
    - 182952g - 13088880h$ \\
\noalign{\hrule height0.8pt}
\end{tabular}
}
\end{center}
\end{table}

\begin{table}[thb]
\caption{Weight enumerator of $S_{120}$}
\label{Tab:120-B}
\begin{center}
{\scriptsize
\begin{tabular}{c|l} 
\noalign{\hrule height0.8pt}
 $i$ & \multicolumn{1}{c}{$B_{i}$} \\
\hline
{12}&$a$ \\
{16}&$17250 - 24a - b$ \\
{20}&$- 315744 + 276a + 22b + c$ \\
{24}&$42581630 - 2024a - 231b - 20c - 64d$ \\
{28}&$6084129120 + 10626a + 1540b + 190c +  1152d + 4096e$ \\
{32}&$475718702550 - 42504a - 7315b - 1140c - 9792d - 65536e - 262144f$ \\
{36}&$18824260734240 + 134596a + 26334b + 4845c +  52224d
     + 491520e + 3670016f + 16777216g$ \\
\noalign{\hrule height0.8pt}
\end{tabular}
}
\end{center}
\end{table}

\begin{rem}\label{rem}
If $C(\cD_{144,72})^\perp$ has minimum weight at least $10$,
then one can show that $C(\cD_{144,72})$ is self-dual by 
an argument similar to that described in above.
\end{rem}

For $m \in \{1,\ldots,6\}$ and $k \in \{m+1,\ldots,3m-1\}$,
the self-duality of $C(\cD_{24m,4k})$ has been verified above.
As a consequence, we have the following:

\begin{prop}\label{prop:CD}
If $m \in \{1,\ldots,6\}$ and $k \in \{m+1,\ldots,3m-1\}$,
then
$C({\cD}_{24m,4k}) = C({\cD}_{24m,24m-4k})$.
\end{prop}
\begin{proof}
It is trivial that
${\cD}_{24m,24m-4k} = \overline{{\cD}_{24m,4k}}$.
For $m \in \{1,\ldots,6\}$ and $k \in \{m+1,\ldots,3m-1\}$,
the codes $C({\cD}_{24m,4k})$ are self-dual (see Table~\ref{Tab:D}).

For $(24m,4k) \in \{(72,16),
(72,32),
(120,32),
(144,32), 
(144,64)\}$,
since the $5$-design 
$\overline{{\cD}_{24m,4k}}$
has odd $r$, $\1 \in C(\overline{{\cD}_{24m,4k}})$.
Consider the remaining cases.
The system of equations (\ref{eq:M}) has
no solution $(n_0,n_2,\ldots)$ consisting of nonnegative
integers for each odd $w$.
By Lemma~\ref{lem:basic} (i),
$\1 \in C(\overline{{\cD}_{24m,4k}})$.
The result follows from Lemma~\ref{lem:CD}.
\end{proof}

By the above proposition, 
for $m \in \{1,\ldots,6\}$ and $k \in \{m+1,\ldots,3m-1\}$,
$C({\cD}_{24m,4k})$ and $C({\cD}_{24m,24m-4k})$ are
self-dual.
In addition, $C({\cD}_{24m,12m})$ are self-dual for $m \in
\{1,\ldots,5\}$.
This completes the proof of Theorem~\ref{thm:main2}.

\section{On the minimum weights}\label{sec:min}

In this section, 
we describe how to determine the minimum weights
given in the third column of Table~\ref{Tab:D}
for the cases denoted by $*$ in Table~\ref{Tab:D}.
For the other cases, 
the minimum weights are determined by 
Lemma~\ref{lem:basic} (i) only.
The result in this section
completes the proof of Theorem~\ref{thm:main1}.

\subsection{$(24m,4k)=(72, 24),( 72, 32)$}

Suppose that $4k \in\{24, 32\}$.
Let ${v}\in C(\cD_{72,4k})^{\perp}$ be a 
vector of weight $w >0$.
For each $w \in \{4,8\}$, 
the system of equations (\ref{eq:M}) has no solution.
From the result in the previous section, 
$C(\cD_{72,4k})$ is a doubly even self-dual code.
By Lemma~\ref{lem:basic} (i),
$C(\cD_{72,4k})$  is a doubly even self-dual code
of length $72$ and minimum weight at least $12$.

By Gleason's theorem (see~\cite{MS73}), 
the weight enumerator of a doubly even self-dual
code of length $n$ can be written as:
\[
\sum_{i=0}^{\lfloor n/24 \rfloor} a_i (x^8+14x^4y^4+y^8)^{n/8-3i}
(x^4y^4(x^4-y^4)^4)^i,
\]
using integers $a_i$.
Hence,
the weight enumerator of $C(\cD_{72,4k})$
can be written as:
\begin{align*}
&
x^{72} + \alpha x^{60} y^{12}
+ (249849  - 12 \alpha) x^{56} y^{16}
+ (18106704  + 66 \alpha) x^{52} y^{20}
\\&
+ (462962955  -  220 \alpha) x^{48} y^{24}
+ (4397342400  + 495 \alpha) x^{44} y^{28}
\\&
+ (16602715899 - 792 \alpha) x^{40} y^{32}
+ (25756721120 +  924 \alpha) x^{36} y^{36}
+ \cdots,
\end{align*}
using a nonnegative integer $\alpha$.
If $\alpha > 0$, then 
the number of codewords of weight $4k=24$ (resp.\ $32$)
is less than
$462962955$ (resp.\ $16602715899$),  
which is
the number of blocks of $\cD_{72,24}$ (resp.\ $\cD_{72,32}$).
Hence, $\alpha=0$.
This means that $C(\cD_{72,4k})$ must be extremal.

\subsection{$(24m,4k)=(96,28), ( 96, 36), ( 96, 44)$}

%

%
The numbers of blocks of $\cD_{96,28}, \cD_{96,36}$ and $\cD_{96,44}$ 
are
\[
18642839520, 4552866656416 \text{ and }65727011639520,
\]
respectively.
If $4k \in\{28, 36,  44\}$, then
it follows from (\ref{eq:M})
that the doubly even self-dual code
$C(\cD_{96,4k})$ has minimum weight at least $12$.
The weight enumerator $\sum_{i=0}^{96}A_i x^{96-i}y^i$
of $C(\cD_{96,4k})$
can be written using
integers $\alpha,\beta$, where
$A_i$ are listed in Table~\ref{Tab:96}.
If there is an integer $i\in \{12,16\}$ with $A_i > 0$,
then 
\[
A_{36} =4552866656416 -4368 A_{12} -192412 A_{16} <4552866656416,
\]
which is the number of the blocks of $\cD_{96,36}$.
This gives a contradiction.
Hence, $A_{12}=A_{16}=0$, then
$\alpha=\beta=0$.
This means that $C(\cD_{96,36})$ is extremal.
Similarly, one can easily show that
$C(\cD_{96,44})$ is extremal, and that
$C(\cD_{96,28})$ is extremal
if $d(C(\cD_{96,28})) \ge 16$.

\begin{table}[thbp]
\caption{Weight enumerator of $C(\cD_{96,4k})$}
\label{Tab:96}
\begin{center}
{\footnotesize
\begin{tabular}{c|l} 
\noalign{\hrule height0.8pt}
 $i$ & \multicolumn{1}{c}{$A_{i}$} \\
\hline
12& $\beta$\\
16& $\alpha + 30 \beta$\\
20& $3217056 - 16 \alpha + 153 \beta $\\
24& $369844880   + 120 \alpha - 1712 \beta $\\
28& $18642839520  - 560 \alpha - 3084 \beta $\\
32& $422069980215  + 1820 \alpha + 69576 \beta $\\
36& $4552866656416 - 4368 \alpha - 323452 \beta $\\
40& $24292689565680 + 8008 \alpha + 842544 \beta $\\
44& $65727011639520 - 11440 \alpha- 1443090 \beta $\\
48& $91447669224080 + 12870 \alpha+ 1718068 \beta $\\
\noalign{\hrule height0.8pt}
\end{tabular}
}
\end{center}
\end{table}

\subsection{$(24m,4k)=(120, 32),(120,40),(120,48),(120,56)$}

The numbers of blocks of $\cD_{120, 32}, \cD_{120,40},
\cD_{120,48}$ and $\cD_{120,56}$ are
\begin{multline*}
475644139425,
397450513031544, 
\\
30531599026535880 \text{ and  }
257257766776517715,
\end{multline*}
respectively.
If $4k \in\{32,40,48,56\}$, then
it follows from (\ref{eq:M})
that the doubly even self-dual code
$C(\cD_{120,4k})$ has minimum weight at least $12$.
The weight enumerator $W_{120,12}=\sum_{i=0}^{120}A_i x^{120-i}y^i$
of $C(\cD_{120,4k})$
can be written using
integers $\alpha,\beta,\gamma$, where
$A_i$ are listed in Table~\ref{Tab:120}.
If there is an integer $i\in \{12,16,20\}$ with $A_i > 0$,
then 
\begin{align*}
A_{56} = &
257257766776517715-1130786592 A_{12} - 16300570 A_{16} 
\\ & - 167960 A_{20}  < 257257766776517715,
\end{align*}
which gives a contradiction.
Hence, $A_{12}=A_{16}=A_{20}=0$, then
$\alpha=\beta=\gamma=0$.
This means that $C(\cD_{120,56})$ is extremal.
Similarly, one can easily show that
$C(\cD_{120,4k})$ is extremal for $4k=40,48$, and that
$C(\cD_{120, 32})$ is extremal 
if $d(C(\cD_{120, 32})) \ge 20$.

\begin{table}[thbp]
\caption{Weight enumerator of $C(\cD_{120,4k})$}
\label{Tab:120}
\begin{center}
{\footnotesize
\begin{tabular}{c|l} 
\noalign{\hrule height0.8pt}
 $i$ & \multicolumn{1}{c}{$A_{i}$} \\
\hline
12& $\gamma$ \\
16& $\beta + 72\gamma$ \\
20& $\alpha + 26\beta + 2004\gamma$ \\
24& $39703755 -20\alpha + 39\beta + 25272\gamma$ \\
28& $6101289120 + 190\alpha - 2148\beta + 100866\gamma$ \\
32& $475644139425 -1140\alpha + 4563\beta - 621288\gamma$ \\
36& $18824510698240+ 4845\alpha + 71058\beta - 3973756\gamma$ \\
40& $397450513031544 -15504\alpha - 613259\beta + 18650088\gamma$ \\
44& $4630512364732800+ 38760\alpha + 2564432\beta + 37650159\gamma$ \\
48& $30531599026535880 -77520\alpha - 7035366\beta - 434682288\gamma$ \\
52& $116023977311397120+ 125970\alpha + 13909076\beta + 1412322984\gamma$ \\
56& $257257766776517715 -167960\alpha - 20667530\beta - 2641019472\gamma$ \\
60& $335200280030755776+ 184756\alpha + 23538216\beta + 3223090716\gamma$ \\
\noalign{\hrule height0.8pt}
\end{tabular}
}
\end{center}
\end{table}


\subsection{$(24m,4k)=(144,36), (144,52),(144,60),(144,68)$}

The numbers of blocks of 
$\cD_{144,36}$,  $\cD_{144,52}$,  $\cD_{144,60}$ and $\cD_{144,68}$ are
\begin{multline*}
9542972508784, 
4686006803807297232, 
\\
170473729066542803616  \text{ and  }
1005386522059285093728, 
\end{multline*}
respectively.
If $4k \in\{36,52,60,68\}$, then
it follows from (\ref{eq:M})
that the doubly even self-dual code
$C(\cD_{144,4k})$ has minimum weight at least $12$.
The weight enumerator $W_{144,12}=\sum_{i=0}^{144}A_i x^{144-i}y^i$
of $C(\cD_{144,4k})$
can be written using
integers $\alpha,\beta,\gamma,\delta$, where
$A_i$ are listed in Table~\ref{Tab:144}.
If there is an integer $i\in \{12,16,20,24\}$ with $A_i > 0$,
then 
\begin{align*}
A_{68} =& 1005386522059285093728 - 1215686694585 A_{12}
\\& 
-16397532256 A_{16} -246582076 A_{20} -2496144 A_{24}\\ 
<& 1005386522059285093728,
\end{align*}
which gives a contradiction.
Hence, $A_{12}=A_{16}=A_{20}=A_{24}=0$, then
$\alpha=\beta=\gamma=\delta=0$.
This means that $C(\cD_{144,68})$ is extremal.
Similarly, 
one can easily show that
$C(\cD_{144,60})$ is extremal, that
$C(\cD_{144,52})$ is extremal if $d(C(\cD_{144,52})) \ge 20$, and that
$C(\cD_{144,36})$ is extremal if $d(C(\cD_{144,36})) \ge 24$.

\begin{table}[thbp]
\caption{Weight enumerator of $C(\cD_{144,4k})$}
\label{Tab:144}
\begin{center}
{\scriptsize
\begin{tabular}{c|l} 
\noalign{\hrule height0.8pt}
 $i$ & \multicolumn{1}{c}{$A_{i}$} \\
\hline
12&$\delta$ \\
16&$\gamma + 114\delta$ \\
20&$\beta + 68\gamma + 5619\delta$ \\
24&$\alpha + 22\beta + 1722\gamma + 154820\delta$ \\
28&$ 481008528-24\alpha - 59\beta + 17684\gamma + 2550861\delta$ \\ 
32&$90184804281+276\alpha - 2152\beta + 11515\gamma + 24260742\delta$ \\ 
36&$9542972508784-2024\alpha + 13286\beta - 881064\gamma + 102200559\delta$ \\ 
40&$559456467836112+10626\alpha + 39788\beta - 982492\gamma - 215159832\delta$ \\ 
44&$18950225255363376-42504\alpha - 861482\beta + 30439192\gamma - 3223863171\delta$ \\
48&$381888573368657355+134596\alpha + 5423416\beta - 58206711\gamma + 568124866\delta$ \\
52&$4686006803807297232-346104\alpha - 21252317\beta - 458108660\gamma +55774876695\delta$ \\ 
56&$35648745873701148864+735471\alpha + 59961226\beta + 3298378982\gamma -82891353732\delta$ \\ 
60&$170473729066542803616-1307504\alpha - 129387017\beta - 11030355684\gamma -479267780119\delta$ \\ 
64&$517692242136399518331 +1961256\alpha + 220368688\beta + 24037485819\gamma +2310638405958\delta$ \\
68&$1005386522059285093728-2496144\alpha - 301497244\beta - 37463473392\gamma -4857003070893\delta$ \\ 
72&$1253789175212713133280 +2704156\alpha + 334387688\beta + 43291346040\gamma +6110981295024\delta$ \\
\noalign{\hrule height0.8pt}
\end{tabular}
}
\end{center}
\end{table}
\bigskip
\noindent {\bf Acknowledgments.}
The author would like to thank Tsuyoshi Miezaki for
verifying the calculations in the proofs
of Propositions~\ref{prop:SD-P1} and
\ref{prop:SD-P2}, independently.
This work is supported by JSPS KAKENHI Grant Number 23340021.



\begin{thebibliography}{99}

\bibitem{AK-book} E.F. Assmus, Jr.\ and J.D. Key,
Designs and Their Codes,
Cambridge Tracts in Mathematics, 103. 
Cambridge University Press, Cambridge, 1992.

\bibitem{AM} E.F. Assmus, Jr.\ and H.F. Mattson, Jr.,
{New $5$-designs},
{\sl J.\ Combin.\ Theory}
{\bf 6} (1969), 122--151.



\bibitem{C-S} J.H.~Conway and N.J.A.~Sloane,
A new upper bound on the minimal distance of self-dual codes,
{\sl IEEE\ Trans.\ Inform.\ Theory}
{\bf 36} (1990), 1319--1333.

\bibitem{CW12} J. Cruz and W. Willems, 
5-designs related to binary extremal self-dual codes of
length $24m$,
Theory and applications of finite fields, 75--80, 
Contemp.\ Math., 579, Amer.\ Math.\ Soc., Providence, RI, 2012. 

\bibitem{D}S.T. Dougherty, 
private communication, July 2005.

\bibitem{H05} M. Harada, 
Remark on a 5-design related to a putative extremal doubly-even 
self-dual $[96,48,20]$ code,
{\sl Des.\ Codes Cryptogr.}
{\bf 37} (2005), 355--358.

\bibitem{HKM04} M. Harada, M. Kitazume and A. Munemasa, 
On a 5-design related to an extremal doubly even self-dual code of
length 72,
{\sl J. Combin.\ Theory Ser.~A}
{\bf 107} (2004), 143--146.

\bibitem{HMM} M. Harada, T. Miezaki and A. Munemasa,
On $t$-designs supported by self-orthogonal codes,
(in preparation).

\bibitem{HMT05} M. Harada, A. Munemasa and V.D. Tonchev, 
A characterization of designs related to an
extremal doubly-even self-dual code of length 48,
{\sl Ann.\ Comb.}
{\bf 5} (2005), 189--198.



\bibitem{MS73}C.L. Mallows and N.J.A. Sloane, 
An upper bound for self-dual codes, 
{\sl Inform.\ Control}
{\bf  22} (1973), 188--200. 


\bibitem{MST}F.J. MacWilliams, N.J.A. Sloane and J.G. Thompson, 
Good self dual codes exist,
{\sl Discrete Math.}
{\bf  3} (1972), 153--162.

\bibitem{M71} N.S. Mendelsohn, 
Intersection numbers of $t$-designs,
In: Studies in Pure Mathematics (presented
to Richard Rado), Academic Press, London, 1971, 145--150.



\bibitem{Sloane72} N.J.A. Sloane,
{Is there a $(72,36)$ $d=16$ self-dual code?}
{\sl IEEE Trans.\ Inform.\ Theory}
{\bf 19} (1973), 251.

\bibitem{T86} V.D. Tonchev,
{A characterization of designs related to the Witt system $S(5,8,24)$},
{\sl Math.\ Z.}
{\bf 191} (1986), 225--230.


\end{thebibliography}
\end{document}